\documentclass{article}
\usepackage{amsmath,amssymb,amsthm,mathrsfs,microtype,cite}
\usepackage{tikz}
\usetikzlibrary{calc}

\newcommand{\CC}{\mathbb{C}}
\newcommand{\QQ}{\mathbb{Q}}
\newcommand{\ZZ}{\mathbb{Z}}

\newcommand{\FF}{\mathbb{F}}
\newcommand{\PP}{\mathbb{P}}

\newtheorem{theorem}{Theorem}

\DeclareMathOperator{\tr}{tr}
\DeclareMathOperator{\GL}{GL}
\DeclareMathOperator{\SL}{SL}
\DeclareMathOperator{\PGL}{PGL}

\title{A rank 18 Waring decomposition of $sM_{\langle 3\rangle}$ with 432
symmetries}
\author{Austin Conner}
\date{November 14, 2017}

\begin{document}
\maketitle
\begin{abstract}
  The recent discovery that the exponent of matrix multiplication is determined
  by the rank of the symmetrized matrix multiplication tensor has invigorated
  interest in better understanding symmetrized matrix multiplication
  \cite{poly17}.  I present an explicit rank 18 Waring decomposition of
  $sM_{\langle 3\rangle}$ and describe its symmetry group.
\end{abstract}

Determining the complexity of matrix multiplication has been a central problem
ever since Strassen showed, in 1969, that one can multiply a pair of $n\times n$
matrices using only $O(n^{2.81})$ arithmetic operations \cite{str69}.  Strassen
defined the \emph{exponent of matrix multiplication} $\omega = \inf \,\{ \tau
\mid {}$matrix multiplication requires $O(n^\tau)$ arithmetic operations$\}$,
and over the following decades a sequence of results has shown $\omega < 2.3729$
\cite{bcrl79,sch81,str87,cw90,sto10,wil11,lg14}. In 2014, however, it was
demonstrated that the only method proving new bounds since 1989, Strassen's
laser method applied to the Coppersmith-Winograd tensor, cannot prove an upper
bound on $\omega$ better than $2.3$ \cite{aflg15}.

It is necessary then to pursue other methods to make further progress.  Strassen
showed that $\omega = \inf \, \{ \tau \mid R(M_{\langle n\rangle}) = O(n^\tau)
\}$, where $M_{\langle n\rangle} \in \smash{\CC^{n^2}}\otimes
\smash{\CC^{n^2}}\otimes \smash{\CC^{n^2}}$ denotes the structure tensor of the
$n\times n$ matrix algebra and $R(M_{\langle n\rangle})$ its tensor rank
\cite{str69}.  One new idea then is to exploit the recent result that this
latter quantity is furthermore equal to $\inf \{ \tau \mid R_s(sM_{\langle
n\rangle}) = O(n^\tau)\}$, where $sM_{\langle n\rangle}$ denotes the result of
symmetrizing $M_{\langle n\rangle}$ and $R_s(sM_{\langle n\rangle})$ its Waring
rank, the smallest $r$ such that $sM_{\langle n\rangle}$ may be written as the
sum of $r$ cubes  \cite{poly17}.  I present a Waring decomposition of
$sM_{\langle 3\rangle}$ and describe its particularly large symmetry group,
which I hope will lead to generalizations to larger $n$.

\pagebreak[1]

Write $V = \CC^{n}$, and define $sM_{\langle n\rangle} \in S^3(V^\ast \otimes
V)$ as the tensor corresponding to the symmetric multilinear map $(A,B,C)
\mapsto \smash{\frac{1}{2}} (\tr(ABC) + \tr(ACB))$.  Consider the 18 matrices
$m_1,\ldots,m_{18}$ below, where $\zeta=e^{2\pi i/3}$ and $a=-2^{-1/3}$.
\nopagebreak
\begin{gather*}
\begin{pmatrix}
1 & -1 & 0 \\
-1 & 1 & 0 \\
0 & 0 & 0
\end{pmatrix} \quad
  \begin{pmatrix}
0 & 0 & 0 \\
0 & 1 & -\zeta \\
0 & -\zeta^{2} & 1
  \end{pmatrix}
  \quad
 \begin{pmatrix}
1 & 0 & -\zeta \\
0 & 0 & 0 \\
-\zeta^{2} & 0 & 1
\end{pmatrix}\\
\begin{pmatrix}
0 & 0 & 0 \\
0 & 1 & -\zeta^{2} \\
0 & -\zeta & 1
\end{pmatrix} \quad
 \begin{pmatrix}
1 & 0 & -1 \\
0 & 0 & 0 \\
-1 & 0 & 1
\end{pmatrix} \quad
 \begin{pmatrix}
1 & -\zeta & 0 \\
-\zeta^{2} & 1 & 0 \\
0 & 0 & 0
\end{pmatrix}\\
\begin{pmatrix}
1 & 0 & -\zeta^{2} \\
0 & 0 & 0 \\
-\zeta & 0 & 1
\end{pmatrix} \quad
 \begin{pmatrix}
1 & -\zeta^{2} & 0 \\
-\zeta & 1 & 0 \\
0 & 0 & 0
\end{pmatrix} \quad
 \begin{pmatrix}
0 & 0 & 0 \\
0 & 1 & -1 \\
0 & -1 & 1
\end{pmatrix}\\ \\
\begin{pmatrix}
a & 0 & 0 \\
0 & a & 0 \\
0 & 0 & a
\end{pmatrix} \quad
 \begin{pmatrix}
0 & 1 & 0 \\
0 & 0 & \zeta \\
\zeta^{2} & 0 & 0
\end{pmatrix} \quad
 \begin{pmatrix}
0 & 0 & 1 \\
\zeta^{2} & 0 & 0 \\
0 & \zeta & 0
\end{pmatrix}\\
\begin{pmatrix}
0 & 1 & 0 \\
0 & 0 & \zeta^{2} \\
\zeta & 0 & 0
\end{pmatrix} \quad
 \begin{pmatrix}
0 & 0 & 1 \\
1 & 0 & 0 \\
0 & 1 & 0
\end{pmatrix} \quad
 \begin{pmatrix}
1 & 0 & 0 \\
0 & \zeta & 0 \\
0 & 0 & \zeta^{2}
\end{pmatrix}\\
\begin{pmatrix}
0 & 0 & 1 \\
\zeta & 0 & 0 \\
0 & \zeta^{2} & 0
\end{pmatrix} \quad
 \begin{pmatrix}
1 & 0 & 0 \\
0 & \zeta^{2} & 0 \\
0 & 0 & \zeta
\end{pmatrix} \quad
 \begin{pmatrix}
0 & 1 & 0 \\
0 & 0 & 1 \\
1 & 0 & 0
\end{pmatrix}
\end{gather*}
\nopagebreak
\begin{theorem}
  $sM_{\langle 3\rangle} = \frac{1}{6} \sum_{i=1}^{18} m_i^{(3)}$. That is, the
  $m_i$ form a rank 18 Waring decomposition of $6sM_{\langle 3\rangle }$.
\end{theorem}

The group $\Gamma = \GL(V^\ast \otimes V)$ naturally acts on $S^3(V^\ast \otimes
V)$, and the stabilizer of $sM_{\langle n\rangle}$ is $\Gamma_{sM_{\langle
n\rangle}} = \PGL(V) \rtimes \ZZ_2$ \cite{ges}. Here the action of $\PGL(V)$ is
induced by its natural action on $V^\ast \otimes V$, and, after choosing a basis
and its dual, $\ZZ_2$ acts as matrix transposition. Such a choice of matrix
transposition is not canonical, but any choice generates the same group modulo
$\PGL(V)$.
 
Notice that any $m_i$ could be replaced by $\zeta m_i$ as these matrices define
the same rank 1 tensor.  To study symmetry, we wish to consider the $m_i$ modulo
this identification. Therefore, write $T_i = \smash{m_i^{(3)}}$, the rank one
symmetric tensors corresponding to the $m_i$, and define the \emph{symmetry of
the decomposition} as the subgroup $\Gamma_S$ of $\Gamma_{sM_{\langle
n\rangle}}$ which leaves the set $S=\{T_1,\ldots,T_{18}\}$ invariant under the
natural induced action on subsets of $S^3(V^\ast \otimes V)$. A symmetry of the
decomposition preserves the set $\{m_1,\ldots,m_{18}\}$ up to powers of $\zeta$.

\begin{theorem} \label{sym}
  The symmetry group $\Gamma_S \cong (\ZZ_3^2\rtimes \SL(2,\FF_3)) \rtimes
  \ZZ_2$, which has order 432.
\end{theorem}
\nopagebreak
The expression in parentheses is the $\PGL(V)$ action, and the $\ZZ_2$ is
generated by matrix transposition with respect to the basis of the
decomposition. To describe the $\PGL(V)$ part of the action, we label each
$3\times 3$ block of matrices with elements of the vector space $\FF_3^2$ as
follows:
\begin{align*}
  &(0,0) \quad (0,1) \quad (0,2) \\
  &(1,0) \quad (1,1) \quad (1,2) \\
  &(2,0) \quad (2,1) \quad (2,2).
\end{align*}
Then $\ZZ_3^2\rtimes \SL(2,\FF_3)$ acts on the first $3\times 3$ block as affine
transformations of $\FF_3^2$ according to this labelling: $\ZZ_3^2$ acts by
translation and $\SL(2,\FF_3)$ acts by linear transformation.  On the second
$3\times 3$ block, $\ZZ_3^2$ acts trivially and $\SL(2,\FF_3)$ again acts as
linear transformations. One can alternatively view the action of $\ZZ_3^2\rtimes
\SL(2,\FF_3)$ on the second $3\times 3$ block as equivalent to the action on its
normal subgroup $\ZZ_3^2$ by conjugation.

The decomposition is also closed under complex conjugation, which acts by
transposing each $3\times 3$ block.  A Galois-type symmetry like this is not in
general in $\Gamma$ and represents another kind of symmetry of decompositions of
tensors defined over $\QQ$. There are no other nontrivial Galois symmetries for
this decomposition, for any such symmetry must be an automorphism of
$\QQ[\zeta]$ fixing $\QQ$. Including complex conjugation as a symmetry of the
decomposition yields a group of order 864.

\begin{proof}[Proof of Theorem \ref{sym}]
  We first describe the representation $\rho : \ZZ_3^2 \rtimes \SL(2,\FF_3) \to
  \PGL(V)$ explicitly by giving the images of a generating set. These elements
  of $\PGL(V)$ can then be observed to act as claimed on the $3 \times 3$
  blocks.  Let $e_r$ and $e_d$ denote the generators of $\ZZ_3^2$ corresponding
  to translation right and down, respectively, and denote elements of
  $\SL(2,\FF_3)$ by their matrices with respect to the standard basis of
  $\FF_3^2$.  Then
\begin{gather*}
  \rho(e_r) = \begin{pmatrix}
0 & 0 & 1 \\
\zeta^2 & 0 & 0 \\
0 & \zeta & 0
\end{pmatrix} \quad
  \rho(e_d) = \begin{pmatrix}
0 & 0 & 1 \\
\zeta & 0 & 0 \\
0 & \zeta^2 & 0
\end{pmatrix} \\
\rho\begin{pmatrix}
  1 & 1\\
  0 & 1
\end{pmatrix} =
\begin{pmatrix}
-\zeta + 1 & \zeta^{2} - 1 & 2 \zeta + 1 \\
\zeta^{2} - 1 & -\zeta + 1 & 2 \zeta + 1 \\
-\zeta + 1 & -\zeta + 1 & -\zeta + 1
\end{pmatrix} \\
\rho\begin{pmatrix}
  0 & 1\\
  -1 & 0
\end{pmatrix} =
\begin{pmatrix}
-\zeta^{2} + 1 & \zeta - 1 & -\zeta^{2} + 1 \\
\zeta - 1 & -\zeta^{2} + 1 & -\zeta^{2} + 1 \\
\zeta - 1 & \zeta - 1 & -2 \zeta - 1
\end{pmatrix} \\
\rho\begin{pmatrix}
  0 & 1 \\
  -1 & -1
\end{pmatrix} =
\begin{pmatrix}
1 & 0 & 0 \\
0 & 1 & 0 \\
0 & 0 & \zeta^{2}
\end{pmatrix}.
\end{gather*}

It remains to show there are no symmetries of the decomposition other than those
claimed.  Name the entries of a $3\times 3$ block by the numbers $1,\ldots,
9$, like on a telephone.  Since all symmetries of $\Gamma_{sM_{\langle
3\rangle}}$ preserve matrix rank, we first observe that any symmetry of the
decomposition must preserve in particular the first $3\times 3$ block.  This,
combined with the fact that there is evidently a matrix transposition in
$\Gamma_S$, shows it is sufficient to check the set of $\PGL(V)$ symmetries of
the first $3\times 3$ block is as claimed.  Call this group $G$. We wish to
show $G= \ZZ_3^2 \rtimes \SL(2,\FF_3)$.  The first block consists of only rank
1 matrices, so they uniquely determine column vectors up to multiplication by
scalars.  Let $H$ denote the symmetry group of the corresponding projective
configuration of points in $\PP^2$.  The vectors corresponding to matrices
$(1,3,4,6)$ are in general linear position, so each element of $H$ determines
at most one element of $\PGL(V)$ which induces it.  Hence, the natural
homomorphism $G\to H$ is injective, so it suffices to show $H \le \ZZ_3^2
\rtimes \SL(2,\FF_3)$.

First we show $\ZZ_3^2 \rtimes \GL(2,\FF_3)$ is the symmetry group of the
combinatorial affine plane consisting of $9$ points and $12$ lines determined by
the points and colinearity relations of the configuration (Figure~\ref{hesse}).
Clearly $\ZZ_3^2$ are symmetries of this configuration. To see that
$\GL(2,\FF_3)$ are also symmetries, notice that we may identify points of the
configuration with the group $\ZZ^2/3\ZZ^2$, and any line through points in the
lattice $\ZZ^2$ projects down to one of our 12 lines when modding out by
$3\ZZ^2$. Then since $\GL(2,\ZZ)$ preserves the lines of $\ZZ^2$, it must be
that $\GL(2,\FF_3)$ preserves the lines of our configuration $\ZZ^2/3\ZZ^2$, as
desired.  Observe that any symmetry is determined by the image of 3 points. For
instance, fixing the image of 1,2, and 5 determines by colinearity the image of
3,8, and 9, which in turn determines the image the remaining 3 points.  Then,
since $\ZZ_3^2 \rtimes \GL(2,\FF_3)$ are all indeed symmetries of the
configuration, the full symmetry group has order at most $9\cdot 8 \cdot 7 =
504$ and contains a subgroup of size $9\cdot 48 = 432$.  The only possibility is
then equality with $\ZZ_3^2\rtimes \GL(2,\FF_3)$, as claimed.

\begin{figure}[h] 
 \begin{center}
   \begin{tikzpicture}{thick}
     \def \eps {.35}
     \def \extend {0.4}
     \foreach \i in {-1,0,1}
     {
       \draw (-1-\eps,\i) -- (1+\eps,\i);
       \draw (\i,-1-\eps) -- (\i,1+\eps);
     }
     \draw ($(-1,-1)+\eps*(180+45:1)$) -- ($(1,1)+\eps*(45:1)$);
     \draw ($(-1,1)+\eps*(90+45:1)$) -- ($(1,-1)+\eps*(270+45:1)$);

     \draw ($(-1,0)+\eps*(90+45:1)$) -- ($(0,-1)+\extend*(270+45:1)$) arc
     (180+45:360+45:{3/4*sqrt(2)}) coordinate (a);
     \draw (a) -- ($(1,1)+\eps*(90+45:1)$);

     \draw ($(0,-1)+\eps*(180+45:1)$) -- ($(1,0)+\extend*(45:1)$) arc
     (-45:180-45:{3/4*sqrt(2)}) coordinate (b);
     \draw (b) -- ($(-1,1)+\eps*(180+45:1)$);

     \draw ($(1,0)+\eps*(-45:1)$) -- ($(0,1)+\extend*(90+45:1)$) arc
     (45:180+45:{3/4*sqrt(2)}) coordinate (b);
     \draw (b) -- ($(-1,-1)+\eps*(270+45:1)$);

     \draw ($(0,1)+\eps*(45:1)$) -- ($(-1,0)+\extend*(180+45:1)$) arc
     (135:180+135:{3/4*sqrt(2)}) coordinate (b);
     \draw (b) -- ($(1,-1)+\eps*(45:1)$);

     \foreach \i in {-1,0,1} {\foreach \j in {-1,0,1} {
       \node[circle, draw, fill=black!50, inner sep=0pt, minimum width=4pt] at (\i,\j) {};
     }}
    
   \end{tikzpicture}
 \end{center}
  \caption{The configuration determined by column vectors of the rank one block.
  This is classically known as the Hesse configuration \cite{hesse}.}\label{hesse}
\end{figure}
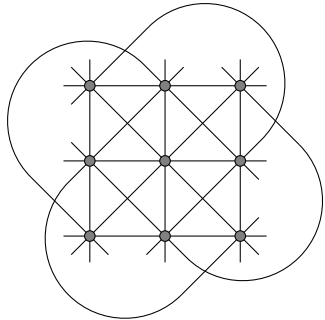

Now we show that the elements of $\ZZ_3^2 \rtimes \GL(2,\FF_3)$ where the second
factor has determinant $-1$ do not induce symmetries of the projective
configuration of points. Because $\ZZ_3^2 \rtimes \SL(2,\FF_3)$ does induce such
symmetries, it suffices to show the failure for only one element.  A convenient
choice is the map $\FF_3^2 \to \FF_3^2$ which interchanges coordinates. The
unique matrix taking the general frame $(1,2,7,8)$ to $(1,4,3,6)$ is
\[
\begin{pmatrix}
0 & -\zeta^{2} & -1 \\
-\zeta^{2} & 0 & -\zeta \\
0 & 0 & \zeta^{2}
\end{pmatrix},
\]
and one readily checks this matrix does not send, e.g. $3$ to any of the other
points. Hence $\ZZ_3^2\rtimes \SL(2,\FF_3) \le G \le H \le \ZZ_3^2\rtimes
\SL(2,\FF_3)$, and the full symmetry group is $(\ZZ_3^2\rtimes
\SL(2,\FF_3))\rtimes \ZZ_2$, as claimed.
\end{proof}

The rank one block of the decomposition consists of orthogonal projections onto
one dimensional subspaces times a factor of two. In this sense, each such matrix
is determined by its column space. We have already seen that these 9 points of
$\PP^2$ form a certain projective configuration (Figure \ref{hesse}).  It is a
classical fact that any set of 9 points in this configuration are the
inflections points of a plane cubic.  Indeed, our configuration is precisely the
inflection points of $x^3+y^3+z^3=0$. Equivalently, it is determined as the
zeros of $x^3+y^3+z^3=0$ and $xyz=0$.

The Waring decomposition presented here was derived from a numerical
decomposition given in \cite{poly17}. I would like to thank Grey Ballard for his
work transforming that numerical decomposition into a sparse numerical one.

\end{document}